\documentclass[10 pt]{amsart}
\usepackage{fullpage}
\usepackage[usenames,dvipsnames]{xcolor}
\usepackage{amsmath,amssymb,amsthm,tensor,enumerate,bm,bbm,setspace}

\newtheorem{thm}{Theorem}[section]

\newtheorem{lemm}[thm]{Lemma}
\newtheorem{cor}[thm]{Corollary}

\newtheorem{prop}[thm]{Proposition}

\newcommand{\om}{\omega}
\newcommand{\vv}{\mathbf{v}}
\newcommand{\xx}{\mathbf{x}}
\newcommand{\ww}{\mathbf{w}}

\newcommand{\In}{\mathbbm{1}}
\newcommand{\yy}{\mathbf{y}}
\newcommand{\pp}{\mathbf{p}}
\newcommand{\qq}{\mathbf{q}}

\newcommand{\C}{\mathbb C}
\newcommand{\La}{\Lambda}
\theoremstyle{definition}

\theoremstyle{remark}

\newtheorem{rema}[thm]{Remark}
\newtheorem*{theorem*}{Theorem}
\theoremstyle{remark}

\numberwithin{equation}{section}

\def\ind{\mathbbm{1}}
\newcommand{\wrt}[1]{\mathrm{d}{#1}}

\def\a{\alpha}

\def\N{\mathbb{N}}
\def\Z{\mathbb{Z}}
\def\ZZ{\mathbb{Z}}

\def\R{\mathbb{R}}
\def\RR{\mathbb{R}}

\def\F{\mathcal{F}}
\def\Nn{\mathcal{N}}
\def\W{\mathcal{W}}
\def\H{\mathcal{H}}
\def\Mm{\mathfrak{M}}
\def\Kk{\mathcal{K}}

\def\cR{\mathcal{R}}
\def\cS{\mathcal{S}}
\def\T{\mathbb{T}}
\def\l{\lambda}
\def\Ga{\Gamma}
\def\Lo{\Lambda_\omega}

\newcommand{\SL}{\operatorname{SL}}
\newcommand{\GL}{\operatorname{GL}}

\newcommand{\aff}{\operatorname{aff}}

\numberwithin{equation}{section}
\title{Ergodic Theory and Diophantine approximation for translation surfaces and linear forms}
\author{Jayadev Athreya \and Andrew Parrish \and Jimmy Tseng}
\address{J.S.A.: Department of Mathematics, University of Washington, Seattle, WA 98102}
\email{jathreya@uw.edu}
\address{A.P.: Department of Mathematics \& Computer Science, Eastern Illinois University, Charleston, IL 61920}
\email{ajnparrish@gmail.com}
\address{J.T.: School of Mathematics, University of Bristol, Bristol BS8 1TW, U.K.}
\email{j.tseng@bristol.ac.uk \quad \quad jimmytseng01@gmail.com}

 \thanks{J.S.A.\ partially supported by NSF grant DMS 1069153, and NSF grants DMS 1107452, 1107263, 1107367 ``RNMS: GEometric structures And Representation varieties" (the GEAR Network), and NSF CAREER grant DMS 1351853.}
 \thanks{J.T. acknowledges the research leading to these results has received funding from the European Research Council under the European Union's Seventh Framework Programme (FP/2007-2013) / ERC Grant Agreement n. 291147 and acknowledges support by the Heilbronn Institute for Mathematical Research.}

\begin{document}

\begin{abstract}
We derive results on the distribution of directions of saddle connections on translation surfaces using only the Birkhoff ergodic theorem applied to the geodesic flow on the moduli space of translation surfaces. Our techniques, together with an approximation argument, also give an alternative proof of a weak version of a classical theorem in multi-dimensional Diophantine approximation due to W. Schmidt~\cite{SchmidtMetrical, SchmidtMetrical2}.  The approximation argument allows us to deduce the Birkhoff genericity of almost all lattices in a certain submanifold of the space of unimodular lattices from the Birkhoff genericity of almost all lattices in the whole space and similarly for the space of affine unimodular lattices.
\end{abstract}

\maketitle

\section{Introduction}  In this paper, we study translation surfaces, linear and affine forms, toral translations, and unimodular lattices of $\RR^d$ using a few simple tools from ergodic theory.  This provides a unified and simplified viewpoint and allows us to derive new results and explain classical ones.

\subsection{Translation Surfaces} A \emph{translation surface} is a pair $(M, \om)$, where $M$ is a Riemann surface and $\om$ is a holomorphic $1$-form. We refer the reader to Zorich~\cite{Zorich} for an excellent survey on translation surfaces. Often, we will use simply $\omega$ to refer to a translation surface. A \emph{saddle connection} on a translation surface $\omega$ is a geodesic $\gamma$ (in the flat metric induced by $\om$) connecting two zeros of $\omega$ (with none in its interior). Moreover, to each saddle connection $\gamma$, one can associate a holonomy vector $\vv_{\gamma} = \int_{\gamma} \om \in \C$. 
 
 The set of holonomy vectors $\La_{\omega}:=\La_{sc}(\om)$ is a discrete subset of $\C \cong \R^2$. Saddle connections arise naturally as special trajectories for billiards in rational-angled polygons; see, for example,~\cite{MT}. As an example, given a unimodular lattice $\La \subset \C$, the associated flat torus $\C/\La$ is a translation surface, and, `marking' the point $0$ as a zero of $\omega$, the set of saddle connections corresponds to the set of \emph{primitive} vectors in $\La$, that is, the set of vectors in $\La$ which are not nontrivial multiples of other vectors in $\La$.
 
 Understanding the geometry of the set $\La_{\omega}$ has been one of the central themes in the study of translation surfaces. Masur~\cite{Masurcount}, Veech~\cite{Veechcount} and Eskin-Masur~\cite{EMcount} proved seminal counting results, showing, respectively, that $$N(\om, R) : = \#(\La_{\om} \cap B(0, R))$$ has upper and lower quadratic upper bounds, quadratic asymptotics on average, and quadratic asymptotics for almost every translation surface. Recently, Eskin-Mirzakhani-Mohammadi~\cite{EMMflat} showed that weak quadratic asymptotics hold for \emph{every} point. 
 
The results of Eskin-Masur~\cite{EMcount} also imply that the directions of saddle connection vectors equidistribute on $S^1$ for almost every surface. The fine-scale distribution has been studied in terms of the gap distribution by Athreya-Chaika~\cite{AChaika}, Athreya-Chaika-Lelievre~\cite{ACL}, and Uyanik-Work~\cite{UW}.  In this paper, we study a counting problem related to the fine-scale distribution properties of $\La_{\om}$, by counting the saddle connections which approximate the vertical to a prescribed degree. Given a translation surface $\omega$, $T >1$, and $b >0$, let
\begin{equation*}
 R_{b,T} (\omega) =\# \left\{ (x,y) \in \Lo : |xy| \leq b \mbox{ and } 1 \leq y < T\right\}
\end{equation*}
denote the function which counts the number of saddle connections in a thinning (hyperbolic) region around the $y$-axis.
 
 \begin{thm}\label{trans}
Let $\mu$ be an ergodic $\SL(2, \R)$-invariant measure on a stratum $\mathcal H$ of translation surfaces. There is a constant $C = C(\mu)$ (known as the \emph{Siegel-Veech constant} of $\mu$) such that for each $b >0$ and $\mu$-a.e. $\omega$, we have

\begin{equation*}
 \lim_{T \rightarrow \infty} \frac{R_{b,T} (\omega)}{2 b C\log T} = 1.
\end{equation*}
\end{thm}

In particular, there is a natural absolutely continuous invariant measure $\mu_{MV}$ on each stratum $\H$ known as the Masur-Veech measure (for which the action of $\SL(2,\R)$ was shown to be ergodic independently by Masur~\cite{Masur} and Veech~\cite{Veech}), and this theorem applies to $\mu_{MV}$-a.e. $\om \in \H$.

\subsubsection{Ergodic Theory, Counting, and Approximation} The proof of Theorem~\ref{trans} relies on ergodic theory on the moduli space of translation surfaces. The idea to use ergodic theory applied to counting problems in this context comes from work of Veech~\cite{Veechcount} and Eskin-Masur~\cite{EMcount} and was inspired by the work on the Quantitative Oppenheim conjecture by Eskin-Margulis-Mozes~\cite{EMMopp}, which used ergodic theory on homogeneous spaces. These results rely on sophisticated equidistribution theorems, and a large motivation for this paper is to show how even the original (Birkhoff) ergodic theorem can yield interesting geometric information when applied in these contexts.

Furthermore, Theorem~\ref{trans} gives, for translation surfaces, an analog of a weak form of Schmidt's theorem (see Section~\ref{subsecDA}), highlighting the well-known connections between translation surfaces and Diophantine approximation.  Our techniques, in fact, give an alternate proof of this weak form of Schmidt's results for a special but important case, which, along with the overall strategy of our proofs, we describe below.  Also, one can apply some of these techniques to the study of the distribution of directions of lattice vectors that arise in Diophantine approximation as done in~\cite{AGT2, AGT1}.

 \subsection{Diophantine Approximation}\label{subsecDA}
 
 Suppose $\psi_{i}(n): \N \to \R, 1 \leq i \leq k$ are non-negative functions and that $\psi(n) = \prod_{i=1}^k \psi_i(n)$ is monotonically decreasing. In 1926, A.~Khintchine showed that 
 \begin{theorem*}[\cite{Khin1926}]
   If  $\sum_{n=1}^\infty \psi(n)$ diverges, then there are infinitely many solutions $(n, p_1, ..., p_k)$ to the system of inequalities
  \begin{equation*}
   \left| x_i n - p_i \right| \leq \psi_i(n),
  \end{equation*}
 for a.e. $(x_1, \cdots, x_k)=:\xx \in \R^k$. If $\sum_{n=1}^\infty \psi(n)$ converges, then there are at most finitely many solutions for a.e. $\xx \in \R^k$.
 \end{theorem*}

 This statement was later refined by W.~Schmidt, who showed in \cite{SchmidtMetrical} (see also~\cite{SchmidtMetrical2} for both linear and affine forms as well as certain polynomials) that the number of solutions of the system of inequalities, \begin{equation*}
   | x_i n - p_i | \leq \psi_i(n),
  \end{equation*} with $ 1 \leq n \leq h$ is on the order of $\sum_{n=1}^h \psi(n)$ while also giving an estimate on the size of the error term.

 % (Note, in addition, that the divergence case of Schmidt's theorem applies to these functions.) 
 
To illustrate the flexibility of our techniques, we show how to give an alternative proof of a weak form of Schmidt's theorem, namely asymptotics for counting the number of solutions without the error term, for the functions, \[\psi_i(n) := \frac 1 {n^{1/k}},\] using only the Siegel mean value theorem, the Birkhoff ergodic theorem, and an approximation argument.  These functions occur naturally in questions of Diophantine approximation, and we note that the divergence case of Schmidt's theorem applies to them. Given $\xx \in \R^k$, $T >1$, let $N(\xx, T)$ denote the number of solutions $(n, \pp) \in \Z \times \Z^k=:\Z^{k+1}$ of the system of inequalities
  \begin{equation*}
   \|n \xx - \pp \|  \leq \frac{1}{|n|^{1/k}}
  \end{equation*}
   with $1 \leq |n| < T$.  Here and below, let $\|\cdot\|$ denote the Euclidean norm (this is not crucial for our results, but helps to streamline our statements).  Let $B_k$ denote the volume of the unit $k$-ball.

 \begin{thm}\label{ThmA} 
  For a.e. $\xx \in \R^k$, $$\lim_{T \rightarrow \infty} \frac{N(\xx, T)}{2 B_k \log T} = 1.$$ \end{thm}

\noindent While this statement is weaker than Schmidt's theorem, the proof given below relies only on the ergodicity of a certain flow on the space of unimodular lattices.

\subsubsection{Homogeneous linear forms} In fact, Schmidt stated his result in terms of linear forms, and our results also apply in this setting. Let $A \in M_{m \times n}(\R)$ be an $m \times n$ matrix, which we view as a system of $m$ linear forms in $n$ variables. A classical Diophantine question is to find approximate integral solutions to the equation $A \yy = \xx$, $\yy \in \R^n, \xx \in \R^m$. In this context, a classical theorem of Dirichlet implies that, for every $A$, there are infinitely many solutions $\pp \in \Z^m, \qq \in \Z^n\backslash \{\boldsymbol{0}\}$ to the inequality $$\|A\qq- \pp\| \le \kappa \|\qq\|^{-\frac n m}$$ where $\kappa$ is a constant depending only on $m$ and $n$.  (If we replace the Euclidean norm $\|\cdot\|$ by the sup norm, then we may take $\kappa$ to be $1$.)  Let $b>0$, $$N(A, b, T) := \# \left\{(\pp, \qq) \in \Z^m \times \Z^n : \|A\qq- \pp\| \le b \|\qq\|^{-\frac n m}, 1 \leq \| \boldsymbol{q} \| < T\right\},$$  $C_k$ denote the surface area of the unit sphere $\mathbb S^{k-1} \subset \R^k$ (e.g., $C_1 =2, C_2 = 2\pi, C_3 = 4\pi$), and recall that $B_k$ denotes the volume of the unit $k$-ball.  Theorem~\ref{ThmA} is, in fact, a special case $(m=k, n=1, b=1)$ of % and $\|\cdot\|$ denotes the Euclidean norm.

\begin{thm}\label{ThmLinear} For each $b >0$ and a.e. $A \in M_{m \times n}(\R)$, we have that $$\lim_{T \rightarrow \infty} \frac{N(A, b, T)}{b^m B_m C_n \log T} = 1.$$  

\end{thm}

%\noindent If we set $m=k, n=1, b=1$ in Theorem~\ref{ThmLinear} and count only over $1\leq q<T$ instead of over $1\leq |q|<T$ (and thus counting exactly half as many), we obtain Theorem~\ref{ThmA}.

\subsubsection{Inhomogeneous linear forms} We can also consider a system of \emph{inhomogeneous linear forms} or, alternatively, \textit{affine forms}: given $A \in M_{m \times n}(\R)$ and $\ww \in \R^m$, we want to approximate integral solutions $(\pp, \qq)$ to the equation $A \qq = \pp + \ww$.  For $b>0$, define $$N(A, \ww, b, T) := \# \left\{(\pp, \qq) \in \Z^m \times \Z^n  : \|A\qq- \pp - \ww\| \le b\|\qq\|^{-\frac n m}, 1 \leq \| \boldsymbol{q} \| < T\right\}.$$

\begin{thm}\label{ThmInhomLinear} For each $b >0$ and a.e. $(A, \ww) \in M_{m \times n}(\R) \times \R^m$, we have that $$\lim_{T \rightarrow \infty} \frac{N(A, \ww,b, T)}{b^m B_m C_n \log T} = 1.$$  

\end{thm}
%\noindent The theorem is a weak form of Schmidt's theorem in~\cite{SchmidtMetrical2} for the special but important case.

%\noindent The theorem extends the weak form of Schmidt's theorem (namely, asymptotics for counting the number of solutions without the error term) from linear forms to affine forms for the special but important case.

\subsection{Lattices}\label{secLattices}  Dynamics on the space of unimodular lattices and Diophantine approximation are strongly linked (see~\cite{KSS} for an introduction and Section~\ref{secApplications} for an instance of this link).  We have results for lattices which are the analogs of our results for linear and affine forms.  Let $m, n \geq 1$, $d=m+n$, and $\La \subset \R^{d}$ be a unimodular lattice, that is, a discrete subgroup of covolume $1$ (see Section~\ref{sec:lattices}). Viewing \[\R^{d} := \R^m \times \R^n,\] we write elements as $\vv = (\xx, \yy), \xx \in \R^m, \yy \in \R^n$. By an abuse of notation, we think of all of these as \emph{column} vectors.  Given $b>0$, define $$R_{b,T}(\La) : = \# \left \{ (\xx, \yy) \in \La: \|\xx\|^m \|\yy\|^n \leq b, 1 \le \|\yy\| <T \right\}.$$

 \begin{thm}\label{ThmB} 
  For each $b >0$ and a.e. unimodular lattice $\La \subset \R^{d}$ (with respect to the Haar measure on the space of unimodular lattices), we have that $$\lim_{T \rightarrow \infty} \frac{R_{b,T}(\La)}{b B_m C_n \log T} = 1.$$ \end{thm}

We also record a statement on affine unimodular lattices: an \textit{affine unimodular lattice} $\La+ \vv$ in $\R^{d}$ is a translate of a unimodular lattice $\La \subset \R^{d}$ by a vector $\vv \in \R^{d}/\La$. We define $R_{b,T}(\La+\vv)$ as above.

\begin{thm}\label{ThmC} 
  For each $b >0$ and a.e. affine unimodular lattice $\La + \vv \subset \R^{d}$ (with respect to the Haar measure on the space of affine unimodular lattices), we have that $$\lim_{T \rightarrow \infty} \frac{R_{b,T}(\La + \vv)}{b B_m C_n \log T} = 1.$$ \end{thm}
  
 A direct computation shows that $b B_m C_n \log T$ is the $d$-dimensional volume of the region $$\mathcal R_{b,T} : =\left \{ (\xx, \yy) \in \R^d: \|\xx\|^m \|\yy\|^n \leq b, 1 \le \|\yy\| <T \right\}$$ where we are counting lattice points.

 \subsection{Toral Translations}\label{sec:toral} Our results for forms can also be interpreted in terms of shrinking target properties (or logarithm laws) for toral translations. Fix $m \geq 1$, and $\boldsymbol{\a} \in \T^m := \R^m/\Z^m$. We consider the dynamical system generated by translation by $\boldsymbol{\a}$, that is, the map $T_{\boldsymbol{\a}}: \T^m \rightarrow \T^m$, where $$T_{\boldsymbol{\a}} \xx = \boldsymbol{\a} + \xx.$$

Let $\|\xx\|_{\Z} = \min\{ \|\xx- \boldsymbol{p}\|: \boldsymbol{p} \in \Z^m\}$. For $b>0$, we define 
\begin{equation*}
 S_{b,N} (\boldsymbol{\a}) := \# \left\{1 \leq  q \leq N : \left\| q\boldsymbol{\a} \right\|_{\Z} < b q^{-1/m} \right\} = \# \left\{1 \leq  q \leq N : T_{\boldsymbol{\a}} ^q (\mathbf{0}) \in B(\mathbf{0}, bq^{-1/m}) \right\}  .
\end{equation*}
where $B(\boldsymbol{0}, r)$ denotes a $\|\cdot\|_{\Z}$-ball in $\T^m$ of radius $r>0$. Similarly, given $\vv \in \T^m$, we define \begin{equation*}
 S_{b, N}(\boldsymbol{\a}, \vv) := \# \left\{ 1 \le q \leq N : \left\| q\boldsymbol{\a}-\vv \right\|_\Z < bq^{-1/m} \right\} = \# \left\{1 \leq q \leq N : T_{\boldsymbol{\a}} ^q (\mathbf{0}) \in B(\mathbf{v}, bq^{-1/m}) \right\} . 
\end{equation*}  
As above, let  $B_{m}$ be the volume of the unit $\|\cdot\|$-ball in $\R^m$.

\begin{cor}\label{theoPowerLogLaw} 
 For each $b >0$ and a.e. $\boldsymbol{\a}$, we have that 
\begin{equation*}
 \lim_{N \to \infty} \frac{S_{b, N} (\boldsymbol{\a})}{b^m B_m \log N}  = 1
\end{equation*} and, for each $b >0$, a.e. $\boldsymbol{\a}$, and a.e. $\vv \in \T^m$, we have that 
\begin{equation*}
 \lim_{N \to \infty} \frac{S_{b, N}(\boldsymbol{\a}, \vv) }{b^m B_m\log N}= 1.
\end{equation*}
\end{cor}

\begin{proof}
The first equality follows by setting $n=1$ in Theorem~\ref{ThmLinear} and counting over $1 \leq  q \leq N$ instead of over $1 \leq  |q| \leq N$ (and, thus, counting exactly half as many).  Likewise, the second equality follows by setting $n=1$ in Theorem~\ref{ThmInhomLinear} and counting over $1 \leq  q \leq N$ instead of over $1 \leq  |q| \leq N$. 
\end{proof}

Corollary~\ref{theoPowerLogLaw} can be regarded as a strengthening of the logarithm law for toral translations (see~\cite[Eqns.~(5.1)~and~(5.2)]{TsengSTP}).  For $m=1$, there is, in addition, a logarithm law for circle rotations for any irrational $\alpha$~\cite[Corollary~1.7]{TsengSTP}, which follows from a certain shrinking target property (see~\cite[Section~1.3]{TsengSTP} for more details).

\subsection{Badly approximable forms and bounded geodesics} Finally, we note that our results for forms, lattices, toral translations, and translation surfaces for the stratum of the flat torus cannot be improved from \textit{almost every} to \textit{every} because of the existence of badly approximable systems of affine forms~\cite[Theorems~1.1 and~1.4]{ET} (see also~\cite{Kl}~and~\cite{T2}).  For general strata, the existence of bounded geodesics~\cite[Theorem~1.2]{KW} (see also~\cite[Theorem~1.3]{CCM}) shows such improvement cannot occur.

\subsection{Strategy of Proof}\label{sec:strategy} There is a common strategy of proof for Theorems~\ref{trans}, \ref{ThmB}, and \ref{ThmC}. Namely, we express the quantities $R_T$ as Birkhoff averages of an appropriate ergodic transformation on a moduli space and apply the Birkhoff ergodic theorem. The limiting integrals on moduli space can be computed by an application of a Siegel (or Siegel-Veech) formula, allowing us to reduce the problem to a volume computation on Euclidean space.  These ideas, along with an appropriate parametrization of $\SL(d, \RR)$ and an approximation argument which allows us to reduce results about unimodular lattices to number theory, appear again in Theorems~\ref{ThmLinear} and~\ref{ThmInhomLinear}.

 \section{The space of unimodular lattices}\label{sec:lattices} Given a unimodular lattice $\La \subset \R^{d}$, we can write $\La = g \Z^{d}$, where $g \in \SL(d, \R)$ is well-defined up to multiplication on the right by elements of $\SL(d, \Z)$. That is, we can identify the space of unimodular lattices with the homogeneous space $X_{d} := \SL(d, \R)/\SL(d, \Z)$.  It is well-known that $\SL(d,\ZZ)$ is a lattice in the unimodular group $\SL(d, \R)$, and, consequently, the measure, with respect to a Haar measure, of any fundamental domain is the same positive finite value.  Since a Haar measure is unique up to a scalar, we may choose the scalar so that the resulting Haar measure of any fundamental domain is unit.  We will use the notation $\mu = \mu_{d}$ for this Haar measure on $\SL(d,\R)$ and for its induced measure on $X_d$.  For clarity, note that we have $\mu_d(X_d)=1$.
 
 \subsection{Mean value formulas}\label{sec:mean} A key ingredient of our proof is the computation of the average number (with respect to $\mu_{d}$) of lattice points in a given subset of $\R^{d}$. This is known as the Siegel mean value theorem, a central result in the geometry of numbers:

\begin{thm}[Siegel's formula, \cite{Siegel, Macbeath}]\label{Siegel}
Let $f \in L^1(\R^{d}, \lambda)$ where $\l$ is the Lebesgue measure on $\R^d$.  Define a function $\widehat{f}$ on $X_{d}$ by
\begin{equation*}
 \widehat{f}(\La) := \sum_{\xx \in \La \setminus \{0\}} f(\xx).
\end{equation*}
Then 
\begin{equation}\label{eqnSiegelForm}
 \int_{X_{d}} \widehat{f} \,\, d\mu = \int_{\R^{d}} f \,\, d\l.
\end{equation}
\end{thm}  
\noindent Note that, if two functions differ in value on a null set, (\ref{eqnSiegelForm}) still holds, so there is no need to distinguish between functions that differ on null sets.  

The space of affine unimodular lattices $Y_{d}$ can be identified with the space \[\SL(d, \R) \ltimes \R^{d} \big{/}\SL(d, \Z) \ltimes \Z^{d}.\]  In other words, it is a fiber bundle over $X_{d}$ with (compact) fiber over a unimodular lattice $\La$ given by the torus $\R^{d}/\La$. It has a natural probability measure $\nu= \nu_{d}$. As above, given $f \in L^1(\R^{d}, \lambda)$, define a function $\tilde{f}$ on $Y_{d}$ by \begin{equation*}
 \tilde{f}(\La + \vv) := \sum_{\xx \in \La +\vv} f(\xx).
\end{equation*}

As above, we have \begin{equation*}
 \int_{Y_{d}} \tilde{f} \,\, d\nu = \int_{\R^{d}} f \,\, d\l,
\end{equation*} which is Siegel's formula for affine unimodular lattices; for a proof, see~\cite[Corollary~5.2]{MS3}.

\subsubsection{Primitive Vectors} There is also a version of Siegel's formula for the transform associated to summing over the set of \emph{primitive} vectors (vectors which are not non-trivial multiples of other vectors in the lattice or, equivalently, are visible from the origin), which has an extra factor $\frac{1}{\zeta(d)}$ on the right hand side (this is the proportion of primitive vectors), and we can obtain analogous results to Theorem~\ref{ThmB} for primitive vectors with a corresponding factor of $1/\zeta(d)$.  

\subsection{Diagonal Flows}\label{subsecDiagFlows}  

Given $m, n 
\geq 1$, let 
\begin{equation*}
 g_t := \left( \begin{array}{cc}
 e^{\frac n m t} I_m& 0\\
  0  & e^{-t} I_n\\
\end{array} 
\right),
\end{equation*} where $I_m$ and $I_n$ denote the $m \times m$ and $n\times n$ identity matrices, respectively.

The one-parameter group $\{g_t\}_{t \in \R}$ acts on the spaces $X_{d}$ and $Y_{d}$ by left multiplication (equivalently, by the linear action on the lattices viewed as subsets of $\R^{d}$). Given $b >0$, let $f_{b}$ be the indicator function of the set  $$\mathcal R_b := \left \{ (\xx, \yy) \in \R^{d}: \|\xx\|^m \|\yy\|^n \leq b, 1 \le \|\yy\| < 2 \right\}.$$ Note that $$g_{-\log 2} \mathcal R_b = \left \{ (\xx, \yy) \in \R^{d}: \|\xx\|^m \|\yy\|^{n} \leq b, 2 \le \|\yy\| < 4 \right\};$$ again, we are thinking of these elements of $\R^{d}$ as column vectors.  Let $\La$ be a unimodular lattice and $\vv$ be an element in $\R^d$.  The key observation in the proof of Theorem~\ref{ThmB} is

\begin{equation}\label{eq:birkhoff} R_{b,2^k}(\La) = \sum_{i=0}^{k-1} \widehat{f}_b (g_{\log 2}^i \La)\end{equation}
and in Theorem~\ref{ThmC} is 
\begin{equation}\label{eq:birkhoff:affine}R_{b,2^k}(\La + \vv) = \sum_{i=0}^{k-1} \tilde{f}_b (g_{\log 2}^i (\La + \vv)).\end{equation}

\subsection{Ergodicity} The key fact that we use follows from the Moore ergodicity theorem and its generalization to the space of affine unimodular lattices:

\begin{theorem*}[\cite{Moore}] The action of $\{g_t\}$ on $X_d$ (and $Y_d$) is ergodic with respect to the Haar measure. In particular, the transformation $g_{\log 2}$ is ergodic.

\end{theorem*}

\noindent For a proof in the case of $Y_d$, see~\cite[Lemma~4.2]{Kl}.  Note that, by Siegel's formula above, $\widehat{f} \in L^1(X_d, \mu)$ and $\tilde{f} \in L^1(Y_d, \nu)$, which allow us to apply the Birkhoff Ergodic Theorem (see, for example, Walters~\cite{Walters}):

\begin{theorem*}[Birkhoff ergodic theorem] Let $T$ be an ergodic measure-preserving transformation of a probability space$(X, \mu)$, and let $f \in L^1(X, \mu)$. Then, for almost every $x \in X$, we have that $$\lim_{N \rightarrow \infty} \frac 1 N \sum_{i=0}^{N-1} f(T^i x) = \int_{X} f d\mu.$$
\end{theorem*}

\subsection{Proof of Theorems~\ref{ThmB} and~\ref{ThmC}}\label{subsecProofThmsBandC} Applying the Birkhoff ergodic theorem to the expressions (\ref{eq:birkhoff}) and (\ref{eq:birkhoff:affine}) and using Siegel's formula and our volume computation from Section~\ref{secLattices}, we obtain, for almost every $\La \in X_d$,

\begin{equation}\label{eq:birkhoff:limit} \lim_{k \rightarrow \infty} \frac{1}{k} R_{b,2^k}(\La) = \lim_{k \rightarrow \infty} \frac{1}{k}  \sum_{i=0}^{k-1} \widehat{f}_b (g_{\log 2}^i \La) = \int_{X_d} \widehat{f}_b d\mu = \lambda(\mathcal R_b)= b B_m C_n \log 2 \end{equation}
and, for almost every $(\La + \vv) \in Y_d$,
\begin{equation}\label{eq:birkhoff:affine:limit} \lim_{k \rightarrow \infty} \frac{1}{k} R_{b,2^k}(\La + \vv) = \lim_{k \rightarrow \infty}  \frac{1}{k} \sum_{i=0}^{k-1} \tilde{f}_b (g_{\log 2}^i (\La+\vv)) = \int_{Y_d} \tilde{f}_b d\nu = \lambda(\mathcal R_b)= b B_m C_n \log 2.\end{equation}
Note that, if $F: [0, \infty) \rightarrow [0, \infty)$ is an increasing function and $ \frac{F(2^k)}{k} \rightarrow \log 2$, then \begin{align}\label{eqnIncFctLogLim}\lim_{T \rightarrow \infty} \frac{F(T)}{\log T} = 1,\end{align} which applied to (\ref{eq:birkhoff:limit}) and (\ref{eq:birkhoff:affine:limit}) yields Theorems~\ref{ThmB} and~\ref{ThmC}, respectively.\qed

\section{Translation Surfaces}  We now apply our technique to translation surfaces, giving a proof of Theorem \ref{trans}.
While above we appeal to the Siegel formula, in the setting of translation surfaces, we use the Siegel-Veech formula. 
\begin{thm}[Siegel-Veech formula, \cite{HODS1B}, pg. 584]\label{SiegelVeech}
Let $f$ be a continuous function of compact support on $\R^2$. Define $\widehat{f}$, a function on the stratum $\H$, by
\begin{equation*}
 \widehat{f}(\om) := \sum_{\vv \in \Lo} f(\vv);
\end{equation*}
here $\Lo$ refers to the set of holonomy vectors on the translation surface $\omega$. Let $\mu$ be a $\SL(2, \R)$-invariant measure on the stratum $\H$. Then there exists $C = C(\mu)$ such that
\begin{equation}\label{eqnSiegelVeech}
 \int_{\H} \widehat{f}(\om) \,\, d\mu(\om) = C \int_{\R^2} f \,\, d\l,
\end{equation}
where $\l$ is the Lebesgue measure on $\R^2$.
\end{thm}

\begin{proof}[Proof of Theorem \ref{trans}:]
Let \[\cR_i := \left\{ (x,y) \in \R^2 : |xy| \leq b \mbox{ and } 2^{i-1} \leq y < 2^i \right\}\] for $i= 1, \ldots$, and define 
\begin{equation*}
g_t := \left( \begin{array}{cc}
 e^{t} & 0\\
  0  & e^{-t}\\
 \end{array} \right) .
\end{equation*}

Noting that $g_{- \log 2} \cR_i = \cR_{i+1}$, we have that 

\begin{align*}
 R_{b, 2^k} (\omega) &= \sum_{i=0}^{k-1} \#( \cR_{i} \cap \Lo ) =  \sum_{i=0}^{k-1} \#(\cR_1 \cap g_{\log 2}^{i} \Lo) = \sum_{i=0}^{k-1} \ind_{\cR_1}(g_{ \log 2}^{i} \Lo).
\end{align*}

By assumption, $\mu$ is an ergodic invariant measure for the action of $g_{\log 2}$ on $\H$; by the Birkhoff ergodic theorem, the Siegel-Veech formula, and our volume computation from Section~\ref{secLattices}, we have that
\begin{align*}
\lim_{k \rightarrow \infty}  \frac{R_{b, 2^k} (\omega) }{k} = \int_{\mathcal{H}} \widehat{\ind}_{\cR_1} \,\, d\mu = C \l( \cR_1 )= 2 b C \log 2.
\end{align*}  
Here $C = C(\mu).$  Note that, to apply the Siegel-Veech formula for the indicator function ${\ind}_{\cR_1}$, we proceed as follows.  Taking an outer approximation given by a tessellation of small closed squares, using the Urysohn lemma, and using the Siegel-Veech formula for continuous functions of compact support, derive an upper bound.  Similarly, for an inner approximation of small open squares, derive a lower bound.  Taking the limit, the desired relation (\ref{eqnSiegelVeech}) holds.  Applying (\ref{eqnIncFctLogLim}) to $F(t) := R_{b, t}(\om)$ yields the desired result. \end{proof}

\section{Applications}\label{secApplications} In this section, we show how to apply the arguments from Section \ref{sec:lattices} to obtain Theorems~\ref{ThmLinear} and \ref{ThmInhomLinear}.  (Note that Theorem~\ref{ThmA} follows immediately from Theorem~\ref{ThmLinear}.)  The main technique here is an approximation argument which may have further applicability to other questions arising from the interface of dynamics on the space of unimodular lattices and Diophantine approximation.  Let \[\Ga := \SL(d, \ZZ).\]

\subsection{Proof of Theorem~\ref{ThmLinear}}\label{subsubsecPTLinearTA} 

Recall that a unimodular lattice $\La \subset \R^{d}$ can be written as $\La = g \Z^{d}$, where $g \in \SL(d, \R)$ is well-defined up to multiplication on the right by elements of $\Ga$.  Consequently, $\La$ can be expressed as $g \Ga$ when it is regarded as an element in $X_d$, and, moreover, there is a bijection between the set of unimodular lattices in $\RR^d$ and $X_d$.  Given a matrix $A \in M_{m \times n}(\R)$, form the associated matrix $$ h_A : = \left(\begin{array}{cc}I_m & -A \\0 & I_n\end{array}\right).$$  Setting $\La_A := h_A \Z^d$, a direct calculation shows that we have, for $\widetilde{b}>0$ and $T >1$,  \begin{align}\label{eqnCountForForms}
 N(A, \widetilde{b}, T) = R_{\widetilde{b}^m,T}(\La_A).\end{align}  We also have that, as $t \rightarrow +\infty$, $$g_{-t} h_A g_{t} \rightarrow I_d.$$ In fact, the set \[\Nn:= \{h_A: A \in M_{m\times n}(\R)\}\] forms the \emph{expanding horospherical} subgroup for $\{g_t\}_{t\geq0}$, and it has a Haar measure, which we denote by $\wrt{\lambda}$.
 
%which still holds when $\La_A$ is replaced with $h_A \gamma \ZZ^d$ for any $\gamma \in \Ga$.  

Recall the definition of $\mu:=\mu_d$.  Define the open set \[\Mm := \left\{\begin{pmatrix} B & A \\ C & D \end{pmatrix} \in \SL(d, \RR) : B \in \GL(m, \RR), A \in  M_{m \times n}(\R), C \in M_{n \times m}(\R), D \in M_{n \times n}(\R) \right\},\] and we note that the set $\SL(d, \RR) \backslash \Mm$ is $\mu$-null.  Restricting the set $\Mm$ to only those elements for which $A$ equals the $m \times n$ matrix with all zero entries, a matrix which we denote by $0$, we obtain \[\H := \left\{\begin{pmatrix} B & 0 \\ C & D \end{pmatrix} \in \SL(d, \RR) : B \in \GL(m, \RR), C \in M_{n \times m}(\R), D \in \GL(n, \R) \right\},\] which, by direct computation, is a subgroup of $\SL(d, \RR)$.  Let $\wrt{\mu_\H}$ denote a left Haar measure on $\H$.  

We now give a parametrization of $\Mm$, which is, essentially, the parametrization given by J.~Marklof in~\cite[Section~3]{Mar10} (see also (2.11) in S.~G.~Dani's paper~\cite{Dani}).

\begin{lemm}\label{lemmParaOfG}
Any element of $\Mm$ can be uniquely expressed as \[\begin{pmatrix} B & 0 \\ C & D \end{pmatrix} h_A\] for some $A \in  M_{m \times n}(\R), B \in \GL(m,\RR)$, $C \in M_{n \times m}(\R)$, and $D \in \GL(n, \RR)$.
\end{lemm}

\begin{proof}  Let \[\begin{pmatrix} \beta & \alpha \\ \gamma & \delta \end{pmatrix}\in \Mm.\]  Then setting $B = \beta$, $C=\gamma$, $A= -\beta^{-1} \alpha$, and $D =\delta- \gamma \beta^{-1} \alpha$ yields the decomposition.  Since the determinant of \[\begin{pmatrix} B & 0 \\ C & D \end{pmatrix}\] is $1$, $D$ is invertible.  This shows that every element of $\Mm$ can be expressed as desired.

If an element of $\Mm$ has two decompositions, then we have \[ \tensor*[]{\begin{pmatrix} B' & 0 \\ C' & D' \end{pmatrix}}{^{-1}} \begin{pmatrix} B & 0 \\ C & D \end{pmatrix}  =h_{A'} h_{-A}.\]  Multiplying out the matrices shows that $A=A'$, $B=B'$, $C=C'$, and $D=D'$.  This gives the desired uniqueness.
 
\end{proof}

Using the parametrization given by Lemma~\ref{lemmParaOfG}, we have that a Haar measure on $\SL(d, \RR)$ is given by \[\wrt{\mu_\H}~\wrt{\lambda},\] which, thus, must be a constant multiple of $\wrt{\mu}$.  By normalizing $\wrt{\mu_\H}$ appropriately, we may assume that the constant is $1$.

Given a subset $\frak{B} \subset \SL(d, \RR),$ define the subset of $X_d$: \[\frak{B} / \Ga := \{g \Ga : g \in \frak{B}\}.\]  Now consider the following submanifold of $\Mm /\Ga \subset X_d$:

\[\W :=  \left\{ h_{A} \Ga : A \in M_{m\times n}(\R)\right\}.\]
We call a unimodular lattice $\La$ {\em Birkhoff generic (for the action of $g_{\log 2}$ with respect to the function $ \widehat{f}_b$)} if \begin{align}\label{eqnTimeSumSpaceSum}
 \lim_{j \rightarrow \infty} \frac{1}{j}  \sum_{i=0}^{j-1} \widehat{f}_b (g_{\log 2}^i \La)= \int_{X_d} \widehat{f}_b d\mu.
\end{align} holds.  By the ergodicity of $g_{-\log 2}$ and the Birkhoff ergodic theorem, (\ref{eqnTimeSumSpaceSum}) holds for $\mu$-almost every $\La$.

To obtain our desired result, we now use an approximation argument to show that almost every lattice in the submanifold $\W$ is Birkhoff generic.  The key difficulty, which we overcome with the proof we now give, is that the functions $\widehat{f}_b$ are unbounded.  Recall that $\SL(d, \RR)$ has a right $\SL(d, \RR)$-invariant metric, using which it follows that the right action of elements of $\Ga$ are by isometries.  Fix a Borel fundamental domain (or strict fundamental domain) $\F$ for the action of $\Ga$ on $\SL(d, \RR)$, namely a Borel set $\F \subset \SL(d, \RR)$ for which $\SL(d, \RR) = \F \Ga$, $\F \gamma_1 \cap \F \gamma_2 = \emptyset$ for all $\gamma_1, \gamma_2 \in \Ga$ such that $\gamma_1 \neq \gamma_2$, $\mu(\partial \F) =0$ where $\partial \F$ is the boundary of $\F$, and, for every compact set $K \subset \SL(d, \RR)$, the set $\{\gamma \in \Ga:  \F \gamma \cap K \neq \emptyset\}$ is finite (see~\cite[Chapter~1,~ Section~(0.40)]{Ma}).  Note that the canonical projection mapping restricted to $\F$ is a bijection and $\mu(\F) =\mu(X_d) =1$.  Moreover, any subset $W$ of $X_d$ lifts via the canonical projection mapping to a disjoint union of isometric subsets and the measure of the intersection of this union with $\F$ is finite and equal to the measure of one of these disjoint subsets.   

%onto its image $\widetilde{X}_d := \{g\Ga: g \in\F\} \subset X_d$ and $\mu(\widetilde{X}_d) =\mu(X_d) =1$

Pick a lattice \[\widetilde{\La}:= h_{\widetilde{A}} \Ga\] in $\W$.  Let $U_{\H}(I_d)$ be a small open ball in $\H$ around the identity element $I_d$ and $U_{\W}(h_{\widetilde{A}})$ be a small open ball in ${\W}$ around $h_{\widetilde{A}}\Ga$.  Then, by Lemma~\ref{lemmParaOfG}, \begin{align}\label{eqnOpenSetU}
 U:=U(\widetilde{\La}):=U_{\H}(I_d) U_{\W}(h_{\widetilde{A}}) \end{align} is a small open set in $X_d$ containing $\widetilde{\La}$.\footnote{For our proof, we can relax the conditions on $U_{\H}(I_d)$ and $U_{\W}(h_{\widetilde{A}})$, if desired, as follows.  We can replace the requirement that these are small by the requirement that the multiplication mapping \[U_{\H}(I_d) \times U_{\W}(h_{\widetilde{A}}) \rightarrow U\] is injective.  We can replace the requirement that $U_{\W}(h_{\widetilde{A}})$ is an open ball by the requirement that it is a measurable set of nonzero measure containing $h_{\widetilde{A}}$.  Finally, we can replace the requirement that $U_{\H}(I_d)$ is an open ball with the requirement that it is a measurable set containing $I_d$ such that every open neighborhood of $I_d$ meets the set in a subset of nonzero measure.}
 
 % Recall that $\SL(d, \RR)$ has a right $\SL(d, \RR)$-invariant metric, using which it follows that the right action of elements of $\Ga$ are by isometries.  Fix a fundamental domain $\F$ for the action of $\Ga$.  Recall that the lattice $\Ga$ acts properly discontinuously on $\SL(d, \RR)$ and, thus, lifting $U$ to $\SL(d, \RR)$ using the canonical projection mapping, we obtain an infinite family of disjoint isometric sets.  The $\mu$-measure of this family that lies in $\F$ is finite and is equal to the $\mu$-measure of one of these disjoint sets.  

We now show, roughly speaking, that almost all lattices of almost all $\H$-translates of the submanifold $\W$ are Birkhoff generic.  Precisely, we apply the following (which is simply an application of Fubini's theorem) to $U$:

% to $U \cap \widetilde{X}_d$:

%\begin{align}\label{eqnFullMeasBGSetCond}
 %\int_{\Nn} \In_{V_M}(h_A)~\wrt{\lambda(A)} = \int_{\Nn} \In_{U_{\W}}(h_A)~\wrt{\lambda(A)} \end{align}

\begin{prop}\label{propFullMeasureBirkGen}
Let  $b >0$ and $f_{b}$ be the indicator function of the set  $$\mathcal R_b := \left \{ (\xx, \yy) \in \R^{d}: \|\xx\|^m \|\yy\|^n \leq b, 1 \le \|\yy\| < 2 \right\}.$$  Let $U := U_{\H} U_{\W}$, parametrized using Lemma~\ref{lemmParaOfG}, be a set of positive $\mu$-measure in $\Mm / \Ga$.   For $\mu_\H$-almost every $M \in U_{\H}$, there exists a subset $V_M \subset U_{\W}$ such that \begin{align}\label{eqnFullMeasBGSetCond} \lambda(V_M) = \lambda(U_{\W}) \end{align} and, for every $h_A\Ga \in V_M$, the lattice \[Mh_{A} \Ga\] is Birkhoff generic with respect to the function $ \widehat{f}_b$. 
\end{prop}

\begin{proof} %By construction, the set $U$ is a subset of $\Mm / \Ga$ and 

We may, without loss of generality, regard $U$ as a subset of $\F$.  Since (\ref{eqnTimeSumSpaceSum}) holds for $\mu$-almost every element in $X_d$, there exists a set $U_{bg} \subset U$ such that every element in $ U_{bg}\Ga$ is Birkhoff generic with respect to the function $ \widehat{f}_b$ and such that \begin{align}\label{eqnBgFullMeas}
 \mu(U_{bg}) = \mu(U). \end{align}  %We use the parametrization given in Lemma~\ref{lemmParaOfG}.
 
For $M \in \H$, define \begin{align*}
U_{bg, M} := \begin{cases}  \{h_A \in U_{\W} : Mh_A \in U_{bg}\} &\text{ if } Mh_A \in U_{bg}\\ \emptyset &\text{ if }  Mh_A \notin U_{bg}\end{cases}.
\end{align*}  Fubini's theorem implies that \[\In_{U_{bg}}(Mh_A)\] is $\lambda$-integrable for $\mu_{\H}$-almost every $M.$  If, for $\mu_{\H}$-almost every $M\in U_{\H}$, (\ref{eqnFullMeasBGSetCond}) holds when we set \[V_M :=U_{bg, M}\Ga,\] then we have proved the proposition.

%\begin{align}\label{eqnFubiniViolation}
% \int_{\Nn} \In_{U_{bg}}(M h_A)~\wrt{\lambda(A)} = \int_{\Nn} \In_{U_{bg,M}}(h_A)~\wrt{\lambda(A)}< \int_{\Nn} \In_{U_{\W}}(h_A)~\wrt{\lambda(A)}. \end{align} 

Otherwise, there exists a subset $V$ of $U_{\H}$ of positive $\mu_{\H}$-measure such that, for every element $M \in V$, we have that \begin{align}\label{eqnFubiniViolation}
\lambda(U_{bg,M}) < \lambda(U_{\W}). \end{align}  Integrating using Fubini's theorem, we have that \begin{align*}
 \mu(U_{bg}) & = \int_{U_{\H}} \int_{U_{\W}} \In_{U_{bg}}(M h_A)~\wrt{\lambda(A)}~\wrt{\mu_\H(M)}\\ &= \int_{U_{\H}\backslash V} \int_{U_{\W}} \In_{U_{bg}}(M h_A)~\wrt{\lambda(A)}~\wrt{\mu_\H(M)} + \int_{ V} \int_{U_{\W}} \In_{U_{bg}}(M h_A)~\wrt{\lambda(A)}~\wrt{\mu_\H(M)}  \\ &= \int_{U_{\H}\backslash V} \int_{U_{\W}} \In_{U_{\H}\backslash V} (M) \In_{U_{bg,M}}(h_A)~\wrt{\lambda(A)}~\wrt{\mu_\H(M)} \\ &\quad \quad \quad \quad \quad \quad\quad \quad \quad\quad \quad \quad+ \int_{ V} \int_{U_{\W}} \In_V(M) \In_{U_{bg,M}}(h_A)~\wrt{\lambda(A)}~\wrt{\mu_\H(M)} \\ & < \int_{U_{\H}} \int_{U_{\W}} \In_{U_{\H}}(M) \In_{U_{\W}}(h_A)~\wrt{\lambda(A)}~\wrt{\mu_\H(M)} = \mu(U),\end{align*} where we have applied (\ref{eqnFubiniViolation}) to obtain the strict inequality.  This contradicts (\ref{eqnBgFullMeas}).  Consequently, the set $V$ cannot exist and the proof of the proposition is complete.
\end{proof}

%By Proposition~\ref{propFullMeasureBirkGen}, the set $L_\infty$ is a $\lambda$-full measure subset of $U_{\W}(h_{\widetilde{A}})$.  

%For the set $U$ in (\ref{eqnOpenSetU}), we have that $U \cap \widetilde{X}_d = \widetilde{U}_{\H} \widetilde{U}_{\W}$ where $\widetilde{U}_{\H} \subset U_{\H}(I_d)$ and $\widetilde{U}_{\W} \subset U_{\W}(h_{\widetilde{A}})$ are subsets for which $\mu_\H(\widetilde{U}_{\H}) = \mu_\H(U_{\H}(I_d))$ and $\lambda(\widetilde{U}_{\W}) = \lambda(U_{\W}(h_{\widetilde{A}}).$  

Let $0<\varepsilon_\ell' \rightarrow 0$ be a given decreasing sequence indexed by $\ell \in \N$.  Applying Proposition~\ref{propFullMeasureBirkGen} to the set $U$ in (\ref{eqnOpenSetU}), we have a full-measure subset of $U_{\H}(I_d)$ from which to pick a sequence of elements \[\left\{\begin{pmatrix} B_\ell & 0 \\ C_\ell & D_\ell \end{pmatrix}\right\}_\ell\] such that \begin{align}\label{eqnNearId}\nonumber
 B_\ell & \xrightarrow[\ell \rightarrow \infty]{} I_m \\ C_\ell & \xrightarrow[\ell \rightarrow \infty]{} 0 \in M_{m,n}(\R) \\\nonumber D_\ell & \xrightarrow[\ell \rightarrow \infty]{} I_n\end{align} and, for which, there exists a sequence of full-measure subsets $V_\ell \subset U_{\W}(h_{\widetilde{A}})$ such that, for every $h_A\Ga \in V_\ell$, the lattice\begin{align}\label{eqnBirkLatticeNearCenter}
  \begin{pmatrix} B_\ell & 0 \\ C_\ell & D_\ell \end{pmatrix} h_A \Ga  \end{align}is Birkhoff generic with respect to $ \widehat{f}_{b}, \widehat{f}_{b +\varepsilon'_\ell},$ and $\widehat{f}_{b- \varepsilon'_\ell}$.  Define \[L_\infty := \bigcap_{\ell\in \N} V_\ell\] and note that $L_\infty \subset U_{\W}(h_{\widetilde{A}})$ and $\lambda(L_\infty) =  \lambda(U_{\W}(h_{\widetilde{A}}))$.  Hence, picking any lattice $h_A \Ga \in L_\infty$ implies that the lattices \[\La_\ell:=\begin{pmatrix} B_\ell & 0 \\ C_\ell & D_\ell \end{pmatrix} h_{A} \Ga\] are Birkhoff generic with respect to $ \widehat{f}_{b}, \widehat{f}_{b +\varepsilon'_\ell},$ and $\widehat{f}_{b- \varepsilon'_\ell}$ for all $\ell$.  

%Using Proposition~\ref{propFullMeasureBirkGen}, we will define a $\lambda$-full measure set $L_\infty \subset U_{\W}(h_{\widetilde{A}})$ as follows.  Let $0<\varepsilon_\ell' \rightarrow 0$ be a given decreasing sequence indexed by $\ell \in \N$.  From $U_{\H}(I_d)$, pick a sequence of elements \[\left\{\begin{pmatrix} B_\ell & 0 \\ C_\ell & D_\ell \end{pmatrix}\right\}_\ell\] such that, for every element $h_A\Ga$ of $L_\infty$, the lattice\begin{align}\label{eqnBirkLatticeNearCenter}
%  \begin{pmatrix} B_\ell & 0 \\ C_\ell & D_\ell \end{pmatrix} h_A \Ga  \end{align}is Birkhoff generic with respect to $ \widehat{f}_{b}, \widehat{f}_{b +\varepsilon'_\ell},$ and $\widehat{f}_{b- \varepsilon'_\ell}$ for all $\ell$ and such that \begin{align}\label{eqnNearId}\nonumber
 %B_\ell & \xrightarrow[\ell \rightarrow \infty]{} I_m \\ C_\ell & \xrightarrow[\ell \rightarrow \infty]{} 0 \in M_{m,n}(\R) \\\nonumber D_\ell & \xrightarrow[\ell \rightarrow \infty]{} I_n.\end{align}  Hence, picking any lattice $h_A \Ga \in L_\infty$ implies that the lattices \[\La_\ell:=\begin{pmatrix} B_\ell & 0 \\ C_\ell & D_\ell \end{pmatrix} h_{A} \Ga\] are Birkhoff generic with respect to $ \widehat{f}_{b}, \widehat{f}_{b +\varepsilon'_\ell},$ and $\widehat{f}_{b- \varepsilon'_\ell}$ for all $\ell$.  
 
 %Let $\gamma \in \Ga$.  
%First consider the case of $\gamma = I_d$, namely the following:

We now approximate lattices $h_A \Ga \in L_\infty$ in the full-measure subset $L_\infty\subset U_{\W}(h_{\widetilde{A}})$ with the lattices $\La_\ell$ as follows.  Recall that the unimodular lattice $g \Ga$ can be expressed as $g \ZZ^d \subset \RR^d$.    Thus, we have \begin{align}\label{eqnGammaIsIdentity}
h_A \ZZ^d, \quad \quad \La_\ell = \begin{pmatrix} B_\ell & 0 \\ C_\ell & D_\ell \end{pmatrix}h_A\ZZ^d. \end{align} Counting lattice points of $h_A \ZZ^d$ in a thinning region \[\widetilde{\cR}_b(\RR^d):= \left \{ (\xx, \yy) \in \RR^d: \|\xx\|^m \|\yy\|^n \leq b, 1 \le \|\yy\| \right\},\] namely finding the cardinality of \[h_A \ZZ^d \cap \widetilde{\cR}_b(\RR^d) =: \widetilde{\cR}_b(h_A \ZZ^d),\] is the same as counting lattice points of $\La_\ell$ in the skewed thinning region \[\cS_{b, \ell} (\RR^d):=\begin{pmatrix} B_\ell & 0 \\ C_\ell & D_\ell \end{pmatrix} \widetilde{\cR}_b(\RR^d)= \left \{\begin{pmatrix} B_\ell & 0 \\ C_\ell & D_\ell \end{pmatrix} (\xx, \yy) \in \RR^d: \|\xx\|^m \|\yy\|^n \leq b, 1 \le \|\yy\| \right\}\] because the lattice points in these two regions are in bijection.  Here, recall, we have used the notation $(\xx, \yy)$ to denote a $d$-column vector.

Since (\ref{eqnNearId}) holds for $B_\ell, C_\ell,$ and $D_\ell$, we have, for every $(\xx, \yy) \in \widetilde{\cR}_b(\RR^d)$, that 

\begin{align}\label{eqnSkewRegionApproxBnds}
 (1 - \varepsilon_\ell) \|\yy\|- \widetilde{\varepsilon}_\ell &\leq \|C_\ell\xx + D_\ell \yy\| \leq (1 + \varepsilon_\ell)\|\yy\| + \widetilde{\varepsilon}_\ell  \\\nonumber (1 - \widehat{\varepsilon}_\ell) \|\xx\|^m\|\yy\|^n  & \leq \| B_\ell \xx\|^m\| C_\ell\xx + D_\ell \yy\|^n \leq(1 + \widehat{\varepsilon}_\ell)\|\xx\|^m\|\yy\|^n \end{align} where $\varepsilon_\ell, \widetilde{\varepsilon}_\ell, \widehat{\varepsilon}_\ell  \rightarrow 0$ as $\ell \rightarrow \infty$.  (Note that $\widetilde{\varepsilon}_\ell$ does not depend on $\xx$ because $\|\xx\|$ is uniformly bounded.)

%, and, for the skewed thinning region, $\|\yy\| \geq 1$

Since we have $\|\xx\|^m \|\yy\|^n \leq b$ and $\|\yy\| \geq 1$, (\ref{eqnSkewRegionApproxBnds}) implies that we may approximate the skewed thinning region $\cS_{b, \ell} (\RR^d)$ with inner and outer thinning regions $\widetilde{\cR}_{b-\varepsilon'_\ell}(\RR^d)$ and $\widetilde{\cR}_{b+\varepsilon'_\ell}(\RR^d)$ up to, possibly, a small precompact set $\Kk_\ell$.  The set $\Kk_\ell$ that we might need to exclude arises as follows.  The inner and outer thinning regions $\widetilde{\cR}_{b-\varepsilon'_\ell}(\RR^d)$ and $\widetilde{\cR}_{b+\varepsilon'_\ell}(\RR^d)$ satisfy the constraint that $\|\yy\|\geq1$, but, for the skewed thinning region $\cS_{b, \ell} (\RR^d)$, the infimum of $\|C_\ell\xx + D_\ell \yy\|$ is close to $1$ but not necessarily $1$.  This follows because $\|C_\ell \xx\| \leq \widetilde{\varepsilon}_\ell$, $\|D_\ell \yy\| \geq (1 - \varepsilon_\ell)\|\yy\|$.  Thus, we must define $\Kk_\ell$ as follows:  \[\Kk_\ell := \begin{cases} \emptyset & \textrm{if } \inf \{\|C_\ell\xx + D_\ell \yy\| : (\xx, \yy) \in \widetilde{\cR}_b(\RR^d) \} \geq 1 \\ \{(\xx, \yy) \in \widetilde{\cR}_b(\RR^d): \|C_\ell\xx + D_\ell \yy\| <1\}& \textrm{if } \inf \{\|C_\ell\xx + D_\ell \yy\| : (\xx, \yy) \in \widetilde{\cR}_b(\RR^d) \} < 1\end{cases}.\]  Therefore, we have the following inclusions \[\widetilde{\cR}_{b-\varepsilon'_\ell}(\RR^d) \subset \begin{pmatrix} B_\ell & 0 \\ C_\ell & D_\ell \end{pmatrix} \left(\widetilde{\cR}_{b}(\RR^d) \backslash \Kk_\ell\right) \subset \widetilde{\cR}_{b+\varepsilon'_\ell}(\RR^d)\] and thus also \[\La_\ell \cap \widetilde{\cR}_{b-\varepsilon'_\ell}(\RR^d) \subset \begin{pmatrix} B_\ell & 0 \\ C_\ell & D_\ell \end{pmatrix} \left(h_A \ZZ^d  \cap \widetilde{\cR}_{b}(\RR^d) \backslash \Kk_\ell\right) \subset \La_\ell \cap \widetilde{\cR}_{b+\varepsilon'_\ell}(\RR^d).\]

Now let us repeat the above, namely counting lattice points, with the thinning region $\widetilde{\cR}_b(\RR^d)$ replaced by its truncation \[\widetilde{\cR}_{b, 2^j}(\RR^d):= \left \{ (\xx, \yy) \in \RR^d: \|\xx\|^m \|\yy\|^n \leq b, 1 \le \|\yy\| < 2^j\right\}.\]  From this, we obtain \begin{align}\label{eqnLatticeRegIncUpToCpct}\La_\ell \cap \widetilde{\cR}_{b-\varepsilon'_\ell, 2^j}(\RR^d) \subset \begin{pmatrix} B_\ell & 0 \\ C_\ell & D_\ell \end{pmatrix} \left(h_A \ZZ^d  \cap \widetilde{\cR}_{b, 2^j}(\RR^d) \backslash \Kk_\ell\right) \subset \La_\ell \cap \widetilde{\cR}_{b+\varepsilon'_\ell, 2^{j+1}}(\RR^d),\end{align}  where the latter inclusion arises because we have that \[\sup \left\{\|C_\ell\xx + D_\ell \yy\| :(\xx, \yy) \in\widetilde{\cR}_{b, 2^j}(\RR^d) \right\} \leq (1 + \varepsilon_\ell)2^j + \widetilde{\varepsilon}_\ell < 2^{j+1}.\]

%the analog of which also holds when $\gamma = I_d$ is replaced by any $\gamma \in \Ga$ in (\ref{eqnGammaIsIdentity}).

Recall the notation and discussion in Section~\ref{subsecDiagFlows}.  Now (\ref{eqnLatticeRegIncUpToCpct}) gives \[R_{b-\varepsilon'_\ell, 2^j}(\La_\ell) \leq R_{b, 2^j}(h_A \ZZ^d) - K_\ell\leq R_{b+\varepsilon'_\ell, 2^{j+1}}(\La_\ell)\] where \[K_\ell:=\#\left(h_A \ZZ^d \cap \Kk_\ell \right).\]  Since any precompact set of $\RR^d$ can only contain a finite number of lattice points, it follows that $K_\ell <\infty.$  Consequently, we have that \[\lim_{j \rightarrow \infty} \frac 1 j R_{b-\varepsilon'_\ell, 2^j}(\La_\ell) \leq \lim_{j \rightarrow \infty}\frac 1 j R_{b, 2^j}(h_A \ZZ^d)\leq \lim_{j \rightarrow \infty}\frac 1 {j+1} R_{b+\varepsilon'_\ell, 2^{j+1}}(\La_\ell).\]  Since, for every $\ell \in \N$, the lattice $\La_\ell \in X_d$ is Birkhoff generic with respect to $ \widehat{f}_{b}, \widehat{f}_{b +\varepsilon'_\ell},$ and $\widehat{f}_{b- \varepsilon'_\ell}$, we immediately have that \[\int_{X_d} \widehat{f}_{b-\varepsilon'_\ell} d\mu \leq \lim_{j \rightarrow \infty}\frac 1 j R_{b, 2^j}(h_A \Ga)\leq \int_{X_d} \widehat{f}_{b+\varepsilon'_\ell} d\mu.\]  Applying Siegel's formula and our volume computation from Section~\ref{secLattices} and letting $\ell \rightarrow \infty$ show that every lattice in $L_\infty$ is Birkhoff generic with respect to $\widehat{f}_{b}$.   As the chosen lattice $\widetilde{\La}$ and small open set $U$ are arbitrary, this proof is valid everywhere, showing that almost every lattice in ${\W}$ is Birkhoff generic with respect to $\widehat{f}_{b}$.  Using this Birkhoff genericity in place of the Birkhoff ergodic theorem in the proof of (\ref{eq:birkhoff:limit}) yields:  for almost every lattice in $h_A \Ga \in {\W}$, \[\lim_{j \rightarrow \infty}\frac 1 j R_{b, 2^j}(h_A \Ga) = b B_m C_n \log 2.\]  Using (\ref{eqnIncFctLogLim}) and setting $b := \widetilde{b}^m$, we have that, for almost every lattice in $h_A \Ga \in {\W}$, \[\lim_{T \rightarrow \infty}\frac {R_{\widetilde{b}^m, T}(h_A \Ga)} {\widetilde{b}^m B_m C_n \log T}=1.\]  Finally, using (\ref{eqnCountForForms}) proves Theorem~\ref{ThmLinear}.  Theorem~\ref{ThmA} is a special case of Theorem~\ref{ThmLinear} with $m=k, n=1$.\qed

%Applying this to the proof in Section~\ref{subsecProofThmsBandC} and using (\ref{eqnCountForForms}) with $b := \widetilde{b}^m$ proves 

\begin{rema}\label{remaFoliateUnstableThruStand}
As a remark, which is not necessary for our proof, we note that Lemma~\ref{lemmParaOfG} implies that \[\{M\W\}_{M \in \H}\] forms a smooth foliation of the open set of full measure $\Mm / \Ga$ and, thus essentially, of $X_d$ also.  The leaves are  $\H$-translates of the submanifold $\W$.
\end{rema}

\subsection{Proof of Theorem~\ref{ThmInhomLinear}}  This is the affine lattice case and will follow by using the argument in Section~\ref{subsubsecPTLinearTA} with a few minor changes.  The changes are as follows.  An affine unimodular lattice $\La + \boldsymbol{v}$ is uniquely determined by a unimodular lattice $\La$ and a vector $\boldsymbol{v} \in \RR^d /\La$.  Using Lemma~\ref{lemmParaOfG} and the fact that $\H$ is a group, we have that any affine unimodular lattice $\La + \vv$ with $\La \in \Mm/\Ga$ can be uniquely written as \[\begin{pmatrix} B & 0 \\ C & D \end{pmatrix}\left( h_A \Ga + \begin{pmatrix}  \vv_1' \\ \vv_2'\end{pmatrix}\right)\] where \[\begin{pmatrix}  \vv_1' \\ \vv_2'\end{pmatrix}\in \RR^d/(h_A \ZZ^d).\]  We note such affine unimodular lattices form an open set whose complement has zero Haar measure.  Also note that the parametrization for the affine lattice case is given by \[\left(\H \ltimes \begin{pmatrix}  \boldsymbol{0} \\ \RR^n \end{pmatrix}\right) \left(\Nn \ltimes  \begin{pmatrix}  \RR^m  \\\boldsymbol{0}\end{pmatrix} \right).\]

Given a pair $(A, \ww) \in M_{m \times n}(\R) \times \R^m$, form the associated affine unimodular lattice $\La_A + \bar{\ww} = h_A \Z^d + \bar{\ww}$, where $\bar{\ww}$ is the vector $(\ww, \mathbf{0}) \in \R^d/( h_A \Z^d)$. As above, a direct calculation shows that we have, for $\widetilde{b}>0$ and $T >1$,  $$N(A, \ww, \widetilde{b}, T) = R_{\widetilde{b}^m,T}(\La_A + \bar{\ww}).$$ 

%for every $\gamma \in \Ga$. 

The analog of $\W$ from Section~\ref{subsubsecPTLinearTA} is \[\W_{\aff}:=\left\{h_{\widetilde{A}} \Ga + \begin{pmatrix}  \widetilde{\ww} \\ \boldsymbol{0}\end{pmatrix}: \widetilde{A} \in M_{m\times n}(\R) \textrm{ and } \widetilde{\ww} \in \RR^m\right\}.\]   
Finally, we note that a thinning region can be approximated by inner and outer thinning regions translated by \[\begin{pmatrix}  \boldsymbol{0}\\ \vv_\ell\end{pmatrix} \in \RR^m \times \RR^n\] where $\vv_\ell \rightarrow \boldsymbol{0}$ as $\ell \rightarrow \infty$.  (This is because the directions extending to infinity of the thinning region is preserved by such translations.)  We note that this approximation is in addition to the approximation of Section~\ref{subsubsecPTLinearTA}.  This shows that almost every affine lattice in $\W_{\aff}$ is Birkhoff generic and, thus, proves Theorem~\ref{ThmInhomLinear}.
\qed

Note the analog of Remark~\ref{remaFoliateUnstableThruStand} applies here.

\subsection*{Acknowledgements}  We are grateful to the referees for their helpful comments.

\bibliography{MetricalTranslationSurfaces}{}

\bibliographystyle{habbrv}

\end{document}